\documentclass{amsart}
\usepackage[utf8]{inputenc}
\usepackage[english]{babel}
\usepackage{amsthm}
\usepackage{amsfonts}
\usepackage{amsmath}
\usepackage{amssymb}
\usepackage{enumerate}
\usepackage{comment}

\newcommand{\R}{\mathbb{R}}
\newcommand{\C}{\mathbb{C}}

\DeclareMathOperator{\rank}{rank}

\DeclareMathOperator{\cone}{cone}
\DeclareMathOperator{\conv}{conv}

\newcommand{\spec}{\text{spec}}

\newcommand\restr[2]{{% we make the whole thing an ordinary symbol
  \left.\kern-\nulldelimiterspace % automatically resize the bar with \right
  #1 % the function
  \vphantom{\big|} % pretend it's a little taller at normal size
  \right|_{#2} % this is the delimiter
  }}

\makeatletter
\def\th@plain{%
  \thm@notefont{}% same as heading font
  \itshape % body font
}
\def\th@definition{%
  \thm@notefont{}% same as heading font
  \normalfont % body font
}
\makeatother

\numberwithin{equation}{section}

\theoremstyle{plain}
\newtheorem{lause}[equation]{Theorem}
\newtheorem{lem}[equation]{Lemma}

\newtheorem{kor}[equation]{Corollary}

\theoremstyle{definition}
\newtheorem{maar}[equation]{Definition}

\theoremstyle{remark}
\newtheorem{huom}[equation]{Remark}

\title{Characterizing matrix monotonicity of fixed order on general sets}
\author{Otte Heinävaara}
\address{University of Helsinki, Department of Mathematics and Statistics, P.O. Box 68 (Gustaf Hällströmin katu 2b), FI-00014 University of Helsinki}
\email{otte.heinavaara@helsinki.fi}

\begin{document}

\begin{abstract}
	We give new characterizations for matrix monotonicity and convexity of fixed order which connects previous characterizations by Loewner, Dobsch, Donoghue, Kraus and Bendat--Sherman. The ideas introduced are then used to characterize matrix monotone functions of arbitrary order on general subsets of the real line.
\end{abstract}


\maketitle

\section{Introduction}

A function $f : (a, b) \to \R$ is said to be $n$-monotone (increasing), if for any two $n \times n$ Hermitian matrices $A, B$, with $aI < A \leq B < bI$\footnote{As usual, matrices are ordered by the Loewner order, partial order induced by the cone of positive semidefinite matrices} one also has $f(A) \leq f(B)$. In a similar vein, $f$ is said to be $n$-convex, if for any $t \in [0, 1]$ and $a I < A, B < b I$ one has $f(t A + (1 - t) B) \leq t f(A) + (1 - t) f(B)$.

In his seminal 1934 paper \cite{Low} Loewner proved that a function, which is $n$-monotone for all $n \geq 1$ is actually real analytic. Moreover, it extends to upper half-plane as a Pick function; function with non-negative imaginary part. One of the steps in the proof of this results is the following characterization of $n$-monotone functions with a matrix of divided differences. Writing
\begin{align*}
	[x, y]_{f} = \begin{cases}
				\frac{f(x) - f(y)}{x - y} & \text{if $x \neq y$} \\
				f'(x) & \text{if $x = y$}
				\end{cases}
\end{align*}
one has
\begin{lause}[Loewner]\label{basic_mon}
A function $f : (a, b) \to \R$ is $n$-monotone (for $n \geq 2$) if and only if $f \in C^{1}(a, b)$ and the Loewner matrix
\begin{align*}
L = ([\lambda_{i}, \lambda_{j}]_{f})_{1 \leq i, j \leq n}
\end{align*}
is positive for any $\lambda_{1}, \lambda_{2}, \ldots, \lambda_{n} \in (a, b)$.\footnote{In this paper we reserve the term \textit{positive} matrix to denote positive semidefinite matrices.}
\end{lause}

In 1937 Dobsch \cite{Dob} gave an alternate characterization for $n$-monotonicity.

\begin{lause}[Dobsch, Donoghue]\label{hankel_mon}
A $C^{2 n - 1}$ function $f : (a, b) \to \R$ is $n$-monotone if and only if the Dobsch matrix
\begin{align*}
M(t) = \left(\frac{f^{(i + j - 1)}(t)}{(i + j - 1)!}\right)_{1 \leq i, j \leq n}
\end{align*}
is positive for any $t \in (a, b)$.
\end{lause}

This characterization has a striking consequence: $n$-monotonicity is a local property, i.e. if $f$ is $n$-monotone on two overlapping open intervals, then it is $n$-monotone on their union. This property, stated by Loewner to be easy (\cite[p. 212, Theorem 5.6]{Low}), was actually used in Dobsch' argument, but was rigorously proved only after almost 40 years by Donoghue in \cite[XIV, Theorem V]{Don}. Donoghue's argument is relatively convoluted, but a simpler approach connecting \ref{basic_mon} and \ref{hankel_mon} directly was developed by the author in \cite{Heina}. It was proven that one has an integral representation connecting Loewner and Dobsch matrices:

\begin{align*}
	L(\Lambda) = (2 n - 1)\int_{-\infty}^{\infty} C^{T}(t, \Lambda)M(t)C(t, \Lambda)I_{\Lambda}(t)d t.
\end{align*}
where $C$ is certain matrix entries of which are polynomials of $t$ and rational functions in $\lambda$'s, and $I$ is some piecewise polynomial compactly supported function (see section \ref{integral_repr_section} for details).

In this paper, a new characterizations for $n$-monotonicity in terms of (higher order) divided differences is given. Recall that divided differences (here denoted by $[\cdot, \ldots, \cdot]_{f}$) are defined recursively by $[\lambda]_{f} = f(\lambda)$ and for pairwise distinct $\lambda_{0}, \lambda_{1}, \lambda_{2}, \ldots, \lambda_{n} \in (a, b)$, divided difference of order $n$ is defined recursively via
\begin{align*}
	[\lambda_{0}, \lambda_{1}, \ldots, \lambda_{n}]_{f} = \frac{[\lambda_{0}, \lambda_{1}, \ldots, \lambda_{n - 1}]_{f} - [\lambda_{1}, \lambda_{2}, \ldots, \lambda_{n}]_{f}}{\lambda_{0} - \lambda_{n}}.
\end{align*}
Equivalently, as one may easily check, one has
\begin{align*}
	[\lambda_{0}, \lambda_{1}, \ldots, \lambda_{n}]_{f} = \sum_{i = 0}^{n} \frac{f(\lambda_{i})}{\prod_{j \neq i} (\lambda_{i} - \lambda_{j})}.
\end{align*}
If $f \in C^{n}(a, b)$, divided difference has continuous extension to all tuples of not necessarily distinct $n + 1$ numbers on the interval (see \cite{Boo}).

\begin{lause}\label{main_monotone}
	A function $f : (a, b) \to \R$ is $n$-monotone, if and only if for any $q \in \R_{n - 1}[x]$ and pairwise distinct $x_{0}, x_{1}, \ldots, x_{2 n - 1} \in (a, b)$ one has
	\begin{align*}
		[x_{0}, x_{1}, \ldots, x_{2 n - 1}]_{f q^{2}} \geq 0.
	\end{align*}
\end{lause}

For $n$-convexity one has similar characterization (see section \ref{convex_sec} for brief historical overview):

\begin{lause}\label{main_convex}
	A function $f : (a, b) \to \R$ is $n$-convex, if and only if for any $q \in \R_{n - 1}[x]$ and pairwise distinct $x_{0}, x_{1}, \ldots, x_{2 n - 1}, x_{2 n} \in (a, b)$ one has
	\begin{align*}
		[x_{0}, x_{1}, \ldots, x_{2 n - 1}, x_{2 n}]_{f q^{2}} \geq 0.
	\end{align*}
\end{lause}

From these two results together with the basic properties of divided differences (namely Lemma \ref{k_tone_local}) it follows immediately that for any $n$ both $n$-monotonicity and $n$-convexity are local properties.

\begin{kor}\label{local_property}
	Let $a < c < b < d$, $n \geq 1$, and $f : (a, d) \to \R$ be such that both $\restr{f}{(a, b)}$ and $\restr{f}{(c, d)}$ are $n$-monotone (resp. $n$-convex). Then also $f$ is $n$-monotone (resp. $n$-convex).
\end{kor}

We say that a function is $k$-tone, if all its divided differences of order $k$ are non-negative.

\begin{kor}\label{regularity}
	Let $n \geq 2$, and $f$ be $n$-monotone ($n$-convex). Then $f$ is $(2 n - 1)$-tone ($(2 n)$-tone) so in particular in $C^{2 n - 3}$ ($C^{2 n - 2}$).
\end{kor}

These results also explain the appearance of the integral representations of \cite{Heina}; these turn out to be relics of the Peano representation of divided differences (see section \ref{integral_repr_section} for details).

In addition to these characterizations, the ideas in the proofs are used to extend the characterization for $n$-monotonicity on general subsets of $\R$. If $F$ is any subset of $\R$, $f : F \to \R$ is $n$-monotone if for any $n \times n$ Hermitian matrices $A \leq B$ with $\spec(A), \spec(B) \subset F$, one has $f(A) \leq f(B)$.

\begin{lause}\label{general_monotone}
	Let $F \subset \R$ and $n \geq 1$. Then $f : F \to \R$ is $n$-monotone, if and only if
	\begin{align*}
		[x_{0}, x_{1}, \ldots, x_{2 k - 1}]_{f q^{2}} \geq 0
	\end{align*}
	for any pairwise distinct $x_{0}, x_{1}, \ldots, x_{2 k - 1} \in F$ and $q \in \C_{k - 1}[x]$; for any $1 \leq k \leq n$.

	Moreover, if $\#F > 2 n$, it suffices to verify the case $k = n$.
\end{lause}

Previously $n$-monotonicity on general sets has been studied only in the case $n = \infty$ (i.e. functions which are $n$-monotone for every $n \geq 1$). Based on the work of {\v{S}}mul'jan \cite{Smulj}, Chandler \cite{Chandler} proved a striking result: if $f$ is $\infty$-monotone on an open set $U$, then it extends as a $\infty$-monotone to the convex hull of $U$. This result was further generalized to general subsets of $\R$ by Donoghue \cite{Don_gen}. We prove in Theorem \ref{interpolation_failure} that such interpolation cannot be done in general for fixed $n$.

See also \cite{Rosen} and \cite{Don_gen2} for different notion of $\infty$-monotonicity for open subsets of $\R$.

This article is largerly based on the author's master's thesis.

\section{Preliminaries}

Divided differences have many useful properties that make them a convenient computational device for tackling matrix functions.

\begin{lem}\label{divided_basic}
	\begin{enumerate}[(i)]
	\item (Frobenius representation for divided differences) For analytic $f$, and suitable\footnote{For our purposes, it is enough to consider entire $f$ and $\gamma$ a circle enclosing the points $z_{1}, z_{2}, \ldots, z_{n}$.} $\gamma$ one has
	\begin{align*}
		[z_{0}, z_{1}, z_{2}, \ldots, z_{n}]_{f} = \frac{1}{2 \pi i} \int_{\gamma} \frac{f(z)}{(z - z_{0}) (z - z_{1})\cdots (z - z_{n})}dz.
	\end{align*}
	\item \label{cond2_divided} (Continuity and mean value theorem) If $f \in C^{n}(a, b)$, then the order $n$ divided differences of $f$ extend continuously to $(a, b)^{n + 1}$ and the extension satisfies a mean value theorem: for any tuple of (not necessarily distinct) real numbers $(x_{i})_{i = 0}^{n} \in (a, b)^{n + 1}$ one has
	\begin{align*}
		[x_{0}, x_{1}, \ldots, x_{n}]_{f} = \frac{f^{(n)}(\xi)}{n!}
	\end{align*}
	for some $\xi \in [\min(x_{i}), \max(x_{i})]$.
	\end{enumerate}
\end{lem}
\begin{proof}
	See for instance \cite{Boo}.
\end{proof}

Recall that functions for which all the divided differences of order $k$ are non-negative are called $k$-tone. If $f \in C^{k}$, then, by Lemma \ref{divided_basic} (ii), $f$ is $k$-tone, if and only if its $k$'th derivative is non-negative. Additionally, $k$-tone functions have the following basic properties.

\begin{lem}\label{k_tone_regularity}
	Let $k \geq 2$. Then $f$ is $k$-tone, if and only if $f \in C^{k - 2}$ and $f^{(k - 2)}$ is convex.
\end{lem}
\begin{proof}
	See for instance \cite{Bullen}.
\end{proof}

\begin{lem}\label{k_tone_local}
	Let $a < c < b < d$ and $f : (a, d) \to \R$ is such that both $\restr{f}{(a, b)}$ and $\restr{f}{(c, d)}$ are $k$-tone, then so is $f$.
\end{lem}

While this local property -result could be certainly easily proven with standard regularization argument from the case $f \in C^{n}$ (which is evident considering \ref{divided_basic} (ii)) it is also follows easily from the following lemma.

\begin{lem}\label{refinement_lemma}
	Let $a < x_{0} < x_{1} < \ldots < x_{k} < b$. Then for any $0 \leq k \leq n$ and any refinement (supersequence) $y_{0} < y_{1} < \ldots < y_{n}$ of $(x_{i})_{i = 0}^{k}$ there exists non-negative numbers $t_{0}, t_{1}, \ldots, t_{n - k}$ such that for any $f : (a, b) \to \R$ one has
	\begin{align*}
		[x_{0}, x_{1}, \ldots, x_{k}]_{f} = \sum_{j = 0}^{n - k} t_{j} [y_{j}, y_{j + 1}, \ldots, y_{j + k}]_{f}.
	\end{align*}
\end{lem}
\begin{proof}
	See for instance \cite{Boo}.
\end{proof}

Also matrix functions can be understood via Cauchy's integral formula. Indeed, if $f$ is analytic, then for suitable $\gamma$ one has
\begin{align}\label{cauchy_matrix}
	f(A) = \frac{1}{2 \pi i} \int_{\gamma} (z I - A)^{-1} f(z) dz.
\end{align}

This observation is particularly fruitful when combined with Lemma \ref{divided_basic} (i): Cauchy's integral formula allows us to interpret linear identities of matrix functions and divided differences as rational functions: if a linear identity, i.e. some linear functional being zero, is proven for all functions of the form $x \mapsto (z - x)^{-1}$, then it automatically holds for any entire function.

\begin{comment}
One can also bootstrap this idea to more general setting.

\begin{lem}\label{identity_lemma}
	Consider linear functional of the form
	\begin{align*}
		F : f \mapsto \sum_{i = 1}^{n} c_{i} f^{(k_{i})}(a_{i}).
	\end{align*}
	If this functional vanishes for infinitely many functions of the form $x \mapsto (z - x)^{-1}$, or for every entire $f$, then it vanishes whenever it makes sense.
\end{lem}
\begin{proof}
	By the assumption $z \mapsto F((z - \cdot)^{-1})$ is a rational function vanishing at infinitely points, hence zero function and the functional is trivial. In entire case apply Hermite interpolation to find polynomial which attains arbitrary values for $f^{(k_{i})}(a_{i})$'s. It follows that the functional is again trivial.
\end{proof}
\end{comment}

There is also a purely formal way to talk about such identities. For any $F \subset \C$, denote by $R^{(k)}(F)$ the rational functions with
\begin{enumerate}
	\item simple poles, all of which lie on $F$
	\item vanishing up to order $k$ at infinity
\end{enumerate}
Then one may define a bilinear pairing between $R^{(1)}(F)$ and $\C^{F}$ (complex functions on $F$) by setting for $f \in \C^{F}$
\begin{align*}
	\left\langle f, \frac{1}{\cdot - a} \right\rangle_{L} = f(a)
\end{align*}
and extending linearly. Then Lemma \ref{divided_basic} (i) rewrites to
\begin{align}\label{formal_divided}
	[x_{0}, x_{1}, \ldots, x_{n}]_{f} = \left\langle f, \frac{1}{(\cdot - x_{0}) \cdots (\cdot - x_{n})}\right\rangle_{L}
\end{align}
and (\ref{cauchy_matrix}) rewrites to
\begin{align}\label{formal_matrix}
	f(A) = \langle f, (z I - A)^{-1} \rangle_{L}.
\end{align}

Note that this pairing also satisfies
\begin{align}\label{formal_shift}
	\langle p f, r\rangle_{L} = \langle f, p r \rangle_{L}
\end{align}
whenever $p$ is a polynomial of degree $k$ and $r \in R^{(1 + k)}(F)$. This follows at once from partial fraction decomposition.

This pairing can be also naturally extended to rational functions with poles of higher order, as long as $f$ is regular enough: if $f$ is $k$ times differentiable at $a$, we may set
\begin{align}\label{formal_shift_der}
	\left\langle f, \frac{1}{(\cdot - a)^{k + 1}}\right\rangle_{L} = \frac{f^{(k)}(a)}{k!}.
\end{align}
Evidently also this extension satisfies (\ref{formal_divided}).

\section{Monotone case}

As usual, denote by $v^{*}$ the adjoint of the map $v : z \to v z$, i.e. the map $w \mapsto \langle w, v\rangle$.

\begin{maar}
	Let us say that a triplet $(A, B, v)$, where $A$ and $B$ are Hermitian $n \times n$ matrices and $v \in \C^{n}$, is a \textbf{projection pair}, if $B - A = v v^{*}$. We say that a projection pair $(A, B, v)$ is \text{strict}, if $v$ is not orthogonal to any eigenvector of $A$.
\end{maar}

We first make some elementary observations on projection pairs.

\begin{lem}
	Let $(A, B, v)$ be a projection pair.
	\begin{enumerate}[(i)]
		\item If the spectra of $A$ and $B$ are disjoint, then $(A, B, v)$ is strict.
		\item If $(A, B, v)$ is a (strict) projection pair, then so is $(-B, -A, v)$.
		\item If $(A, B, v)$ is strict, then $A$ and $B$ have no repeated eigenvalues.
	\end{enumerate}
\end{lem}
\begin{proof}
	\begin{enumerate}[(i)]
		\item Note that if $(A, B, v)$ is not strict and an eigenvector $e_{i}$ of $A$ is orhtogonal to $v$, then $e_{i}$ is also an eigenvector of $B$ with the same eigenvalue.
		\item The non-strict claim is trivial, so we merely need to prove that if $v$ is orthogonal to any eigenvector of $B$, then it is orthogonal to an eigenvector of $A$. But this was exactly our argument in the previous part.
		\item If $A$ has a repeated eigenvalue, then some eigenvector with this eigenvalue is orhtogonal to $v$. The claim on the spectrum of $B$ follows similarly after applying the previous part.
	\end{enumerate}
\end{proof}

For $q \in \C[x]$ denote by $q^{*}$ the polynomial with conjugated coefficients and denote also
\begin{align*}
	N(q) := q q^{*}.
\end{align*}

The beef of Theorem \ref{main_monotone} lies in the following observation.

\begin{lem}\label{main_identity}
	Let us denote
	\begin{align*}
		R_{n, +}(a, b) = \left\{\frac{N(q)(z)}{\prod_{i = 0}^{2 n - 1} (z - x_{i})} | a < x_{0} < \ldots < x_{2 n - 1} < b, q \in \C_{n - 1}[x]\right\}.
	\end{align*}
	Then 
	\begin{align*}
	R_{n, +}(a, b) = &\{\langle ((z I - B)^{-1} - (z I - A)^{-1}) w, w \rangle | w \in \C^{n}, a I < A, B < b I, \\
		&(A, B, v) \text{ is a strict projection pair for some $v$}\} \\
	\end{align*}
\end{lem}
\begin{proof}
	``$\supset$": As one easily checks, for any two invertible matrices $C, D$ one has
	\begin{align}
		C^{-1} - D^{-1} = D^{-1} ((D - C) + (D - C) C^{-1} (D - C)) D^{-1}.
	\end{align}
	Plugging in $C = z I - B$ and $D = z I - A$ for a strict projection pair $(A, B, v)$, and taking $\langle (\cdot) w, w \rangle$ on both simplifies to
	\begin{align*}
		\langle ((z I - B)^{-1} - (z I - A)^{-1}) w, w \rangle = \langle (z I - A)^{-1} v, w\rangle \langle (z I - A)^{-1} w, v\rangle \left(1 + \langle (z I - B)^{-1} v, v\rangle \right).
	\end{align*}
	Write $ \langle (z I - A)^{-1} v, w\rangle = q(z)/\det(z I - A)$; note that $q \in \C_{n - 1}[x]$ and $\langle (z I - A)^{-1} w, v\rangle = q^{*}/\det(z I - A)$. We next claim that as $w$ ranges over $\C^{n}$, $q$ ranges over $\C_{n - 1}[x]$. If $(e_{i})_{i = 1}^{n}$ is an eigenbasis of $A$, with $\lambda_{i}(A)$'s as the respective eigenvalues, we have
	\begin{align*}
		q(z) = \det(z I - A) \langle (z I - A)^{-1} v, w\rangle = \sum_{i = 1}^{n} \left(\prod_{j \neq i} (z - \lambda_{j}(A))\right) \langle v, e_{i}\rangle \langle e_{i}, w\rangle =: \sum_{i = 1}^{n} p_{i}(z) \langle e_{i}, w\rangle.
	\end{align*}
	But since $(A, B, v)$ is strict, $p_{i}(\lambda_{j}(A)) \neq 0$, if and only if $i = j$, so $p_{i}$'s are linearly independent and hence span $\C_{n - 1}[x]$.

	We now have
	\begin{align*}
		\langle ((z I - B)^{-1} - (z I - A)^{-1}) w, w \rangle = \frac{N(q)}{\det(z I - A)^{2}} \left(1 + \langle (z I - B)^{-1} v, v\rangle \right)
	\end{align*}
	Note that all poles of the left-hand side are simple. By choosing $q = 1$, this implies that $\left(1 + \langle (z I - B)^{-1} v, v\rangle \right)$ vanishes at the spectrum of $A$ and hence
	\begin{align}\label{char_pol}
		1 + \langle (z I - B)^{-1} v, v\rangle = \det(z I - A)/\det(z I - B).
	\end{align}
	But again by the simplicity of the poles of the left-hand side, spectra of $A$ and $B$ are disjoint, and we are hence done with the first inclusion.

	``$\subset$": Calculations of the previous inclusion imply that we just need to verify that as $(A, B, v)$ ranges over all strict projection pairs on $(a, b)$, the spectra of $A$ and $B$ range over all tuples of $2 n$ distinct numbers on $(a, b)$. Note that if we manage to find a projection pair with prescribed spectra of $2 n$ distinct numbers, it is automatically strict: non-strict projection pairs have repeated eigenvalues.

	To this end first choose $a < x_{0} < \ldots < x_{2 n - 1} < b$ and $B$ with $\{x_{2 i + 1} | 0 \leq i \leq n - 1\}$ as spectrum (and arbitary eigenspaces). By (\ref{char_pol}) we need to find $v$ such that $r(z) := \left(1 + \langle (z I - B)^{-1} v, v\rangle \right)$ has $x_{2 i}$'s as zeros. As $v$ ranges over $\C^{n}$, (by moving to eigenspaces of $B$) one sees that $r$ ranges over all rational functions with $r(\infty) = 1$ and simple poles at $\{x_{2 i + 1} | 0 \leq i \leq n - 1\}$ with non-negative residues. Since $x_{2 i}$'s and $x_{2 i + 1}$'s interlace, $\prod_{i = 0}^{n} (z - x_{2 i})/(z - x_{2 i + 1})$ is such function, so we are done.
\end{proof}

\begin{kor}\label{main_corollary}
	Let $f : (a, b) \to \R$. Then $f(A) \leq f(B)$ for any strict projection pair $(A, B, v)$ with $a I < A \leq B < b I$, if and only if
	\begin{align}\label{main_condition}
		[x_{0}, x_{1}, \ldots, x_{2 n - 1}]_{f N(q)} \geq 0
	\end{align}
	for any $a < x_{0} < x_{1} < \ldots < x_{2 n - 1} < b$ and $q \in \C_{n - 1}[x]$.
\end{kor}
\begin{proof}
	Taking $\langle f, \cdot \rangle_{L}$ on both sets in Lemma \ref{main_identity} and applying (\ref{formal_divided}), (\ref{formal_matrix}) and (\ref{formal_shift}) reveals that
	\begin{align*}
		& \left\{[x_{0}, x_{1}, \ldots, x_{2 n - 1}]_{f N(q)} | a < x_{0} < x_{1} < \ldots < x_{2 n - 1} < b \text{ and } q \in \C_{n - 1}[x]\right\} \\
		=& \left\{ \left\langle f N(q), \frac{1}{(\cdot - x_{0}) \cdots (\cdot - x_{2 n - 1})} \right\rangle_{L} |\, a < x_{0} < x_{1} < \ldots < x_{2 n - 1} < b \text{ and } q \in \C_{n - 1}[x]\right\} \\
		=& \left\{ \left\langle f , \frac{N(q)}{(\cdot - x_{0}) \cdots (\cdot - x_{2 n - 1})} \right\rangle_{L} |\, a < x_{0} < x_{1} < \ldots < x_{2 n - 1} < b \text{ and } q \in \C_{n - 1}[x]\right\} \\
		=& \left\{ \left\langle f , \langle (((\cdot) I - B)^{-1} - ((\cdot) I - A)^{-1}) w, w \rangle \right\rangle_{L} | w \in \C^{n},\right. \\
		& \left.a I < A \leq B < b I \text{ and } (A, B, v) \text{ is a strict projection pair} \right\} \\
		=& \left\{ \langle (f(B) - f(A)) w, w \rangle | w \in \C^{n},\right. \\
		& \left.a I < A \leq B < b I \text{ and } (A, B, v) \text{ is a strict projection pair} \right\} \\
	\end{align*}
\end{proof}

To prove Theorem \ref{main_monotone} we need two more observations.

\begin{lem}\label{span_lemma}
	Fix $n \geq 1$ and Let $f : (a, b) \to \R$ be such that $[x_{0}, x_{1}, \ldots, x_{2 n + 1}]_{f N(q)} \geq 0$ for any $a < x_{0} < \ldots < x_{2 n + 1} < b$ and $q \in \C_{n}[x]$. Then $[x_{0}, x_{1}, \ldots, x_{2 n - 1}]_{f N(q)} \geq 0$ for any $a < x_{0} < \ldots < x_{2 n - 1} < b$ and $q \in \C_{n - 1}[x]$.
\end{lem}
\begin{proof}
	By \ref{formal_divided} we just need to verify that
	\begin{align*}
		R_{n, +}(a, b) \subset \cone(R_{n + 1, +}(a, b)) := \left\{\sum_{i = 1}^{m} t_{i} v_{i} | t_{i} \geq 0, v_{i} \in R_{n + 1, +}(a, b) \right\}
	\end{align*}
	for any $n \geq 1$.

	So fix any $a < x_{0} < \ldots < x_{2 n - 1} < b$ and $q \in \C_{n - 1}[x]$. Take also any $y_{0} < y_{1} < y_{2} \in (a, b)$ disjoint from $x_{i}$'s. Now
	\begin{align*}
		\frac{N(q)}{\prod_{i = 0}^{2 n - 1} (\cdot - x_{i})} &= \frac{y_{2} - y_{1}}{y_{2} - y_{0}}\frac{N((\cdot - y_{0}) q)}{(\cdot - y_{0}) (\cdot - y_{1})\prod_{i = 0}^{2 n - 1} (\cdot - x_{i})} + \frac{y_{1} - y_{0}}{y_{2} - y_{0}}\frac{N((\cdot - y_{2})q)}{(\cdot - y_{1}) (\cdot - y_{2})\prod_{i = 0}^{2 n - 1} (\cdot - x_{i})} \\
		&\in \cone(R_{n + 1, +}(a, b)).
	\end{align*}
\end{proof}

\begin{lem}\label{polynomial_lemma}
	Let $n$ be a positive integer and $p$ be a complex polynomial. Then the following are equivalent.
	\begin{enumerate}
		\item $p$ is non-negative on $\R$ and of degree at most $2 n$.
		\item $p = q_{1}^{2} + q_{2}^{2}$ for some $q_{1}, q_{2} \in \R_{n}[x]$.
		\item $p = N(q)$ for some $q \in \C_{n}[x]$.
	\end{enumerate}
\end{lem}
\begin{proof}
	See for instance \cite{Prestel}.
\end{proof}

\begin{proof}[Proof of Theorem \ref{main_monotone}]
	``$\Rightarrow$": This follows immediately from Corollary \ref{main_corollary}.

	``$\Leftarrow$": Take any $a I < A \leq B < b I$. Since by the spectral theorem one can write $A = A_{0} \leq A_{1} \leq \ldots \leq A_{n - 1} \leq A_{n} = B$ where $\rank(A_{i + 1} - A_{i}) \leq 1$, we may assume that $(A, B, v)$ is a projection pair for some $v \in \C^{n}$. We may also assume that $v$ is not orthogonal to any eigenvector of $A$. Indeed, if this is not the case, then $A$ and $B$ share eigenvector and eigenvalues and the discussion can be reduced to $(n - 1)$-dimensional space, the orthocomplement of this common eigenspace. We are now done by induction on $n$ and Lemma \ref{span_lemma}.

	Finally, the case of strict projection pair follows immediately from Corollary \ref{main_corollary} and Lemma \ref{polynomial_lemma}.
\end{proof}

\begin{kor}\label{mon_reg}
	If $f$ is $n$-monotone, then $f \in C^{2 n - 3}$, $f^{(2 n - 3)}$ is convex and $f^{(2 n - 1)}$ exists and is non-negative almost everywhere.
\end{kor}
\begin{proof}
	This follows immediately from Theorem \ref{main_monotone} and Lemma \ref{k_tone_regularity}.
\end{proof}

\begin{kor}\label{mon_loc}
	For any positive integer $n$, $n$-monotonicity is a local property.
\end{kor}
\begin{proof}
	This follows immediately from Theorem \ref{main_monotone} and Lemma \ref{k_tone_local}.
\end{proof}

\section{Connection to Loewner and Dobsch matrices}

Theorem \ref{main_monotone} is closely related to both Theorems \ref{basic_mon} and \ref{hankel_mon}; they can both be understood as a special case of the divided difference condition in Theorem \ref{main_monotone} with particular choices of $x_{i}$'s.

\begin{lause}\label{monotone_full}
	Let $n \geq 2$, $(a, b) \subset \R$, $f : (a, b) \to \R$. Then the following are equivalent:
	\begin{enumerate}
		\item \label{cond1} $f$ is $n$-monotone.
		\item \label{cond2} For any $a < x_{0} < x_{1} < \ldots < x_{2 n - 1} < b$ and $q \in \R_{n - 1}[x]$
		\begin{align*}
			[x_{0}, x_{1}, \ldots, x_{2 n - 1}]_{f q^2} \geq 0.
		\end{align*}
		\item \label{cond3} For any $a < x_{0} < x_{1} < \ldots < x_{2 n - 1} < b$ and $q \in \C_{n - 1}[x]$
		\begin{align*}
			[x_{0}, x_{1}, \ldots, x_{2 n - 1}]_{f N(q)} \geq 0.
		\end{align*}
		\item \label{cond4} $f \in C^{1}(a, b)$ and for any $a < x_{1} < x_{2} < \ldots < x_{n} < b$ and $q \in \C_{n - 1}[x]$
		\begin{align*}
			[x_{1}, x_{1}, x_{2}, x_{2}, \ldots, x_{n}, x_{n}]_{f N(q)} \geq 0.
		\end{align*}
		\item \label{cond5} $f \in C^{1}(a, b)$ and for any $a < x_{1} < x_{2} < \ldots < x_{n} < b$ and the Loewner matrix
		\begin{align*}
		([x_{i}, x_{j}]_{f})_{1 \leq i, j \leq n}
		\end{align*}
		is positive.
		\item \label{cond55} $f \in C^{1}(a, b)$ and for any $a < x_{1} < x_{2} < \ldots < x_{n} < b$ and the extended Loewner matrix
		\begin{align*}
		([x_{1}, x_{2}, \ldots, x_{i}, x_{1}, x_{2}, \ldots, x_{j}]_{f})_{1 \leq i, j \leq n}
		\end{align*}
		is positive.
		\item \label{cond6} $f \in C^{2 n - 3}(a, b)$, $f^{(2 n - 3)}$ is convex and $(f N(q))^{(2 n - 1)}$ exists almost everywhere and is non-negative, for every $q \in \C_{n - 1}[x]$.
		\item \label{cond7} $f \in C^{2 n - 3}(a, b)$, $f^{(2 n - 3)}$ is convex and the matrix
		\begin{align*}
			\left(\frac{f^{(i + j - 1)}(t)}{(i + j - 1)!}\right)_{1 \leq i, j \leq n}
		\end{align*}
		exists and is positive almost everywhere.
	\end{enumerate}
\end{lause}
\begin{proof}
	$(\ref{cond1}) \Leftrightarrow (\ref{cond2}) \Leftrightarrow (\ref{cond3})$ follows from the proof of Theorem \ref{main_monotone}.

	$(\ref{cond3}) \Rightarrow (\ref{cond4})$ follows from Theorem \ref{main_monotone} and Corollary \ref{mon_reg}.

	To prove $(\ref{cond4}) \Rightarrow (\ref{cond3})$ we show what if $f \in C^{1}(a, b)$, then for any $a < x_{0} < x_{1} < \ldots < x_{2 n - 1} < b$ we may find $a < y_{1} < y_{2} < \ldots < y_{n} < b$ such that
	\begin{align*}
		[x_{0}, x_{1}, \ldots, x_{2 n - 1}]_{f} = [y_{1}, y_{1}, \ldots, y_{n}, y_{n}]_{f}.
	\end{align*}
	To this end apply first mean value theorem \ref{divided_basic} (ii) to function $g : x \mapsto [x, x_{2}, x_{3}, \ldots, x_{2 n - 1}]_{f}$ to find $x_{0} < y_{1} < x_{1}$ with
	\begin{align*}
		[x_{0}, x_{1}, \ldots, x_{2 n - 1}]_{f} = [x_{0}, x_{1}]_{g} = [y_{1}, y_{1}]_{g} = [y_{1}, y_{1}, x_{2}, \ldots, x_{2 n - 1}]_{f}
	\end{align*}
	Now apply the same trick inductively to $x \mapsto [x, y_{1}, y_{1}]_{f} \in C^{1}(y_{1}, b)$. This implies the claim.

	To prove $(\ref{cond4}) \Leftrightarrow (\ref{cond5})$, note that for any $a < x_{1} < x_{2} < \ldots < x_{n} < b$ and $c_{1}, c_{2}, \ldots, c_{n} \in \C$ one has
	\begin{align*}
		\sum_{i, j = 1}^{n} \frac{c_{i} \overline{c_{j}}}{(\cdot - x_{i}) (\cdot - x_{j})} &= \frac{\left(\sum_{i = 1}^{n} c_{i} \prod_{j \neq i} (\cdot - x_{j})\right) \left(\sum_{i = 1}^{n} \overline{c_{i}} \prod_{j \neq i} (\cdot - x_{j})\right)}{\prod_{i = 1}^{n}(\cdot - x_{i})^{2}} =: \frac{N(q)}{\prod_{i = 1}^{n}(\cdot - x_{i})^{2}}.
	\end{align*}
	As in the proof of Lemma \ref{main_identity} as $c_{i}$'s range over $\C$, $q$ ranges over $\C_{n - 1}[x]$. Applying $\langle f, \cdot\rangle_{L}$ to both sides gives the claim.

	For $(\ref{cond4}) \Leftrightarrow (\ref{cond55})$, after noting
	\begin{align*}
		& \sum_{i, j = 1}^{n} \frac{c_{i} \overline{c_{j}}}{(\cdot - x_{1}) (\cdot - x_{2}) \cdots (\cdot - x_{i}) (\cdot - x_{1}) (\cdot - x_{2}) \cdots (\cdot - x_{j})} \\
		=& \frac{\left(\sum_{i = 1}^{n} c_{i} \prod_{j > i} (\cdot - x_{j})\right) \left(\sum_{i = 1}^{n} \overline{c_{i}} \prod_{j > i} (\cdot - x_{j})\right)}{\prod_{i = 1}^{n}(\cdot - x_{i})^{2}} =: \frac{N(q)}{\prod_{i = 1}^{n}(\cdot - x_{i})^{2}}
	\end{align*}
	the rest follows as in the $(\ref{cond4}) \Leftrightarrow (\ref{cond5})$.

	$(\ref{cond3}) \Rightarrow (\ref{cond6})$ follows immediately from Theorem \ref{main_monotone} and Corollary \ref{mon_reg}.

	To prove $(\ref{cond6}) \Rightarrow (\ref{cond3})$ first fix $q \in \C_{n - 1}[x]$. In view of Corollary \ref{mon_reg} we should prove that $f N(q) \in C^{2 n - 3}$ and $g := (f N(q))^{(2 n - 3)}$ is convex. The regularity part is clear from the assumption. For the convexity of $g$ note that since $f^{2n - 3}$ is convex, it is absolutely continuous locally (on every closed subinterval of $(a, b)$) and hence is also $g$. It hence suffices to verify that $g'$ (which exist almost everywhere) is increasing. Note that
	\begin{align*}
		g'(x) = f^{(2 n - 2)}(x) N(q)(x) + (2 n - 2) f^{(2 n - 3)}(x) (N(q))'(x) + (\text{something $C^{1}$}).
	\end{align*}
	We claim that the divided differences of $g'$ are bounded locally from below. Indeed, since $f^{(2 n - 3)}(x)$ is convex, its divided differences are locally bounded, and so are those of $(2 n - 2) f^{(2 n - 3)}(x) (N(q))'(x)$. The last term is clear. For the first term note that since $f^{(2 n - 2)}$ is increasing
	\begin{align*}
		[x, y]_{f^{(2 n - 2)} N(q)} = N(q)(x) [x, y]_{f^{(2 n - 2)}} + [x, y]_{N(q)} f^{(2 n - 2)}(y) \geq [x, y]_{N(q)} f^{(2 n - 2)}(y),
	\end{align*}
	which is locally bounded from below.

	It follows that on every closed subinteval of $(a, b)$ $g'$ can be written as a sum of an absolutely continuous function and increasing singular function. But since by the assumption derivative of $g'$ is non-negative, the absolutely continuous part is increasing and so is $g'$.

	Finally, to prove $(\ref{cond6}) \Leftrightarrow (\ref{cond7})$ note that for any $t$ and $c_{1}, c_{2}, \ldots, c_{n} \in \C$ one has
	\begin{align*}
		\sum_{i, j = 1}^{n} \frac{c_{i} \overline{c_{j}}}{(\cdot - t)^{i} (\cdot - t)^{j}} = \frac{\left(\sum_{i = 1}^{n} c_{i} (\cdot - t)^{n - i} \right) \left(\sum_{i = 1}^{n} \overline{c_{i}} (\cdot - t)^{n - i}\right)}{(\cdot - t)^{2 n}} =: \frac{N(q)}{(\cdot - t)^{2 n}}.
	\end{align*}
	Yet again, when $c_{1}, c_{2}, \ldots, c_{n} \in \C$ range over $\C$, $q$ ranges over $\C_{n - 1}[x]$. Applying $\langle f, \cdot\rangle_{L}$ to both sides gives the claim.
\end{proof}

Loewner originally proved in \cite{Low} $(\ref{cond1}) \Leftrightarrow (\ref{cond5})$ in the previous theorem. In \cite{Dob} Dobsch proved $(\ref{cond5}) \Rightarrow (\ref{cond7})$ and in \cite{Don} Donoghue proved the converse $(\ref{cond7}) \Rightarrow (\ref{cond5})$; our argument for $(\ref{cond6}) \Rightarrow (\ref{cond3})$ is a variant of the Donoghue's argument. The rest of the conditions are new.

\section{Convex case}\label{convex_sec}

The study of $n$-convex functions was initiated by Kraus in \cite{Kraus}, where the following result was proven:

\begin{lause}[Kraus]\label{basic_con}
	A function $f : (a, b) \to \R$ is $n$-convex (for $n \geq 2$) if and only if $f \in C^{2}(a, b)$ and the Kraus matrix
	\begin{align*}
	Kr = ([\lambda_{i}, \lambda_{j}, \lambda_{0}]_{f})_{1 \leq i, j \leq n}
	\end{align*}
	is positive for any tuple of numbers $(\lambda_{i})_{i = 1}^{n} \in (a, b)^{n}$ and $\lambda_{0} \in (\lambda_{i})_{i = 1}^{n}$.
\end{lause}

Akin to the monotone case, also this result has a local variant.

\begin{lause}\label{hankel_con}
A $C^{2 n}$ function $f : (a, b) \to \R$ is $n$-convex if and only if the Hankel matrix
\begin{align*}
	K(t) = \left(\frac{f^{(i + j)}(t)}{(i + j)!}\right)_{1 \leq i, j \leq n}
\end{align*}
is positive for any $t \in (a, b)$.
\end{lause}

The case $n = \infty$ of the previous was proved in \cite{Ben}; The ``only if" -direction was proved in \cite{Tom}, where also the full result was conjectured and proved in the case $n = 2$ (see also \cite{Tom2}); the full Theorem was proved in \cite{Heina}.

Many of the ideas used for monotone functions translate directly to convex case. Main obstacle is that tools like Lemma \ref{main_identity} are off-limits: one cannot simultaneusly control eigenvalues of three different maps. One can however still connect Theorems \ref{basic_con}, \ref{hankel_con} and \ref{main_convex} in the following manner.

\begin{lause}\label{convex_full}
	Let $k \geq 2$, $(a, b) \subset \R$ and $f : (a, b) \to \R$. Then the following are equivalent.
	\begin{enumerate}
		\item \label{cond1_c} $f$ is $n$-convex.
		\item \label{cond2_c} For any $a < x_{0} < x_{1} < \ldots < x_{2 n} < b$ and $q \in \R_{n - 1}[x]$
		\begin{align*}
			[x_{0}, x_{1}, \ldots, x_{2 n}]_{f q^2} \geq 0.
		\end{align*}
		\item \label{cond3_c} For any $a < x_{0} < x_{1} < \ldots < x_{2 n - 1} < b$ and $q \in \C_{n - 1}[x]$
		\begin{align*}
			[x_{0}, x_{1}, \ldots, x_{2 n}]_{f N(q)} \geq 0.
		\end{align*}
		\item \label{cond4_c} $f \in C^{2}(a, b)$ and for any $a < x_{1} < x_{2} < \ldots < x_{n} < b$, $1 \leq l \leq n$ and $q \in \C_{n - 1}[x]$
		\begin{align*}
			[x_{1}, x_{1}, x_{2}, x_{2}, \ldots, x_{n}, x_{n}, x_{l}]_{f N(q)} \geq 0.
		\end{align*}
		\item \label{cond5_c} $f \in C^{2}(a, b)$ and for any $a < x_{1} < x_{2} < \ldots < x_{n} < b$, $x_{0} \in (a, b)$ and $q \in \C_{n - 1}[x]$
		\begin{align*}
			[x_{1}, x_{1}, x_{2}, x_{2}, \ldots, x_{n}, x_{n}, x_{0}]_{f N(q)} \geq 0.
		\end{align*}
		\item \label{cond6_c} $f \in C^{2}(a, b)$ and for any $a < x_{1} < x_{2} < \ldots < x_{n} < b$, $1 \leq l \leq n$ the matrix
		\begin{align*}
			([x_{i}, x_{j}, x_{l}]_{f})_{1 \leq i, j \leq n}
		\end{align*}
		is positive.
		\item \label{cond7_c} $f \in C^{2}(a, b)$ and for any $a < x_{1} < x_{2} < \ldots < x_{n} < b$, $x_{0} \in (a, b)$ the matrix
		\begin{align*}
			([x_{i}, x_{j}, x_{0}]_{f})_{1 \leq i, j \leq n}
		\end{align*}
		is positive.
		\item \label{cond8_c} $f \in C^{2 n - 2}(a, b)$, $f^{(2 n - 2)}$ is convex and $(f N(q))^{(2 n)}$ exists almost everywhere and is non-negative, for every $q \in \C_{n - 1}[x]$.
		\item \label{cond9_c} $f \in C^{2 n - 2}(a, b)$, $f^{(2 n - 2)}$ is convex and the matrix
		\begin{align*}
			\left(\frac{f^{(i + j)}(t)}{(i + j)!}\right)_{1 \leq i, j \leq n}
		\end{align*}
		exists and is positive almost everywhere.
	\end{enumerate}
\end{lause}
\begin{proof}
	Proof is almost the same as that of Theorem \ref{monotone_full}. Main difference is that as there is no direct result connecting $(\ref{cond1_c})$, $(\ref{cond2_c})$ and $(\ref{cond3_c})$, Theorem \ref{basic_con} (proven in \cite{Kraus}; see also \cite{Hiai}) is used to connect $(\ref{cond1_c})$ instead to $(\ref{cond6_c})$.
\end{proof}

\section{Explanation of the integral representations}\label{integral_repr_section}

In \cite{Heina} following two results were proved.

\begin{lause}\label{monotone_integral_formula}
	Let $a < x_{1} < x_{2} < \ldots < x_{n} < b$. Then the following identity holds for any $f \in C^{2 n - 1}(a, b)$:
	\begin{align*}
		L(f, (x_{i})_{i = 1}^{n}) = \int_{-\infty}^{\infty} C(t, (x_{i})_{i = 1}^{n})^{T} M_{n}(t, f) C(t, (x_{i})_{i = 1}^{n}) I(t, (x_{i})_{i = 1}^{n}) d t.
	\end{align*}
	Here $C$ satisfies
	\begin{align*}
		\sum_{i = 1}^{n} C(t, (x_{i})_{i = 1}^{n})_{i, j} x^{i - 1} = \prod_{j \neq i} (1 + x (t - x_{i}))
	\end{align*}
	and $I(t, (x_{i})_{i = 1}^{n})$ is certain non-negative piecewise polynomial function supported on $[x_{1}, x_{n}]$. (See \cite{Heina} for details).
\end{lause}
\begin{lause}\label{convex_integral_formula}
	Let $a < x_{1} < x_{2} < \ldots < x_{n} < b$ and $a < x_{0} < b$. Then the following identity holds for any $f \in C^{2 n}(a, b)$:
	\begin{align*}
		Kr(f, (x_{i})_{i = 1}^{n}, x_{0}) = \int_{-\infty}^{\infty} C(t, (x_{i})_{i = 1}^{n})^{T} K_{n}(t, f) C(t, (x_{i})_{i = 1}^{n}) J_{x_{0}}(t, (x_{i})_{i = 1}^{n}) d t.
	\end{align*}
	Here $C$ is and in Theorem \ref{monotone_integral_formula} and $J_{x_{0}}(t, (x_{i})_{i = 1}^{n})$ is certain non-negative piecewise polynomial function supported on the convex hull of $x_{i}$'s i.e. $[\min(x_{1}, x_{0}), \max(x_{n}, x_{0})]$. (See \cite{Heina} for details).
\end{lause}

Such representation make it clear that positivity of Loewner and Kraus matrices are implied by the positivity of the respective Hankel matrices.

It turns out that one can understand these identities as special cases of the following general fact about divided differences.

\begin{lem}[Peano representation for divided differences]\label{peano_representation}
	Let $a < x_{0} \leq x_{1} \leq \ldots \leq x_{n} < b$ such that $x_{0} \neq x_{n}$. Then there exists a piecewise polynomial non-negative function $w = w(\cdot, (x_{i})_{i = 0}^{n})$ such that for any $f \in C^{n}(a, b)$ one has
	\begin{align*}
		[x_{0}, x_{1}, \ldots, x_{n}]_{f} = \int_{a}^{b} \frac{f^{(n)}(t)}{n!} w(t) dt.
	\end{align*}
\end{lem}
\begin{proof}
	See for instance \cite{Boo}.
\end{proof}

In turns out that in Theorems \ref{monotone_integral_formula} and \ref{convex_integral_formula} one has
\begin{align}\label{weight_correspondence}
	I(\cdot, (x_{i})_{i = 1}^{n}) = w(\cdot, (x_{1}, x_{1}, \ldots, x_{n}, x_{n})) \\
	J_{x_{0}}(\cdot, (x_{i})_{i = 1}^{n}) = w(\cdot, (x_{0}, x_{1}, x_{1}, \ldots, x_{n}, x_{n})). \nonumber
\end{align}

With Lemma \ref{peano_representation} and (\ref{weight_correspondence}) one can give short proofs for Theorems \ref{monotone_integral_formula} and \ref{convex_integral_formula}.

\begin{proof}[Proof of Theorem \ref{monotone_integral_formula}]
	Note that one has
	\begin{align*}
		L(f, (x_{i})_{i = 1}^{n}) = \left\langle f, \left(\frac{1}{(\cdot - x_{i}) (\cdot - x_{j})}\right)_{i, j}^{n}\right\rangle_{L}
	\end{align*}
	and
	\begin{align*}
		 & C(t, (x_{i})_{i = 1}^{n})^{T} K_{n}(t, f) C(t, (x_{i})_{i = 1}^{n}) \\
		 =& \left\langle f, C(t, (x_{i})_{i = 1}^{n})^{T} \left(\frac{1}{(\cdot - t)^{i}(\cdot - t)^{j}}\right)_{i, j}^{n}C(t, (x_{i})_{i = 1}^{n})\right\rangle_{L} \\
		 =& \left\langle f, \left(\left(\frac{1}{(\cdot - t)^{i}}\right)_{i = 1}^{n} C(t, (x_{i})_{i = 1}^{n})\right)^{T}\left(\left(\frac{1}{(\cdot - t)^{i}}\right)_{i = 1}^{n}  C(t, (x_{i})_{i = 1}^{n})\right)\right\rangle_{L} \\
		 =& \left\langle f, \left(\frac{1}{(\cdot - t)^{2}} \prod_{i' \neq i} \left(1 + \frac{t - x_{i'}}{\cdot - t}\right) \prod_{j' \neq j} \left(1 + \frac{t - x_{j'}}{\cdot - t}\right) \right)_{i, j = 1}^{n} \right\rangle_{L} \\
		 =& \left(\left\langle f, \frac{1}{(\cdot - t)^{2 n}} p_{i} p_{j}\right\rangle_{L}\right)_{i, j}^{n} \\
		 =& \left(\frac{(f p_{i} p_{j})^{(2 n - 1)}(t)}{(2 n - 1)!}\right)_{i, j}^{n}
	\end{align*}
	where $p_{i} = \prod_{j \neq i}(\cdot - x_{j})$. Consequently by Lemma \ref{peano_representation} one has
	\begin{align*}
		&\int_{-\infty}^{\infty} C(t, (x_{i})_{i = 1}^{n})^{T} M_{n}(t, f) C(t, (x_{i})_{i = 1}^{n}) I(t, (x_{i})_{i = 1}^{n}) d t \\
		= & \left(\int_{x_{0}}^{x_{n}} \frac{(f p_{i} p_{j})^{(2 n - 1)}(t)}{(2 n - 1)!} w(t, (x_{0}, x_{0}, \ldots, x_{n}, x_{n})) dt\right)_{i, j}^{n} \\
		= & \left([x_{1}, x_{1}, \ldots, x_{n}, x_{n}]_{f p_{i} p_{j}}\right)_{i, j}^{n} \\
		= & \left(\left\langle f p_{i} p_{j}, \frac{1}{(\cdot - x_{1})^2 \cdots (\cdot - x_{n})^{2}}\right\rangle_{L}\right)_{i, j = 1}^{n} \\
		= & \left\langle f, \left(\frac{1}{(\cdot - x_{i}) (\cdot - x_{j})}\right)_{i, j = 1}^{n} \right\rangle_{L},
	\end{align*}
	as desired.
\end{proof}

\begin{proof}[Proof of Theorem \ref{convex_integral_formula}]
	Proof is almost identical to that of Theorem \ref{monotone_integral_formula}.
\end{proof}

\begin{huom}
Note that $C(t, (x_{i})_{i = 1}^{n})$ is the matrix of change of basis between
\begin{align*}
	\left( \frac{1}{(\cdot - t)^{i}} \right)_{i = 1}^{n} \text{ and } \left( \frac{p_{i}(t)}{(\cdot - t)^{n}}\right)_{i = 1}^{n}.
\end{align*}
\end{huom}

\section{Matrix monotone functions on general sets}

While much of the discussion of matrix monotonicity on open interval $(a, b)$ generalizes directly to general sets, Theorem \ref{main_monotone} doesn't quite make it through, for rather obvious reason. If $F$ is set with less than $2 n$ elements then the condition on $n$-monotonicity is void, but not every function is $n$-monotone. Similar failure can happen if $|F| = 2 n$ and $n > 1$: take $F = \{x_{0}\}_{i = 1}^{2 n - 1}$ for some $x_{0} < x_{1} < \ldots < x_{2 n - 1}$ and define a function $f : F \to \R$ with $f(x_{i}) = \delta_{i, 1}$. This function is definitely not $n$-monotone but, as one easily checks, it satisfies (\ref{main_condition}). These turn out to be the only bad things that can happen.

Only two observations we used before made use of the fact that we were working on open intervals:
\begin{itemize}
	\item In Lemma \ref{span_lemma} we made use of $3$ extra points; if $F$ has at least $2 n + 1$ points, the Lemma is applicable. This explains the ``Moreover" -part of the statement of Theorem \ref{general_monotone}. One could of course make more precise statements on the behaviour when $|F|$ is smaller than $2 n$ (even for small $|F|$ there is no need to verify (\ref{main_condition}) for more than two $k$'s etc.).

	\item In the proof of Theorem \ref{main_monotone} we reduced the discussion to the case of projection pairs by cooking up a chain of projection pairs $(A_{i}, A_{i + 1}, v_{i})$ with $A_{0} = A$ and $A_{n} = B$. Replacing this argument is the main obstacle in the proof of Theorem \ref{general_monotone}.
\end{itemize}

\begin{proof}{Proof of theorem \ref{general_monotone}}
	$``\Rightarrow"$: Simply use Lemma \ref{main_identity} on subspaces of dimensions $1 \leq k \leq n$.

	$``\Leftarrow"$: We prove that for any $A \leq B$ and $w \in \C^{n}$ the term $\langle (f(B) - f(A)) w, w\rangle$ can written as a sum of terms of the form \ref{main_condition} for various $k$. Recall that if $0 \leq A \leq B$, then $B^{-1} \leq A^{-1}$. This implies that for any $A \leq B$ and $w \in \C^{n}$ we have
	\begin{align}\label{poly_decomp}
		\langle ((z I - B)^{-1} - (z I - A)^{-1}) w, w\rangle = \frac{q(z)}{(z - x_{1}) \cdots (z - x_{k})},
	\end{align}
	where $x_{1} < \ldots < x_{k}$ are the distinct eigenvalues of $A$ and $B$; and $q \in \R_{k - 1}[x]$ is non-negative $[x_{k}, \infty)$ and non-negative/non-positive on $(-\infty, x_{1})$ for even/odd $k$, respectively. But any such polynomial can be written in particularly simple form.
	\begin{lem}\label{polynomial_ab_lemma}
		Let $-\infty < a \leq b < \infty$ and $q$ a real polynomial of degree $k$.
		\begin{enumerate}
			\item If $k$ is even and $q$ is non-negative outside $(a, b)$, $q$ can be written as a positive linear combination of polynomials of the form
			\begin{align*}
				N(\tilde{q}) & \text{ for $\tilde{q} \in \C_{\frac{k}{2}}[x]$} \hspace{1 cm} \text{and}\\
				(\cdot - a) (\cdot - b) N(\tilde{q})  & \text{ for $\tilde{q} \in \C_{\frac{k}{2} - 1}[x]$}.
			\end{align*}
			\item If $k$ is odd and $q$ is non-negative on $[b, \infty)$ and non-positive on $(-\infty, a]$, then it can written as a positive linear combination of polynomials of the form
			\begin{align*}
				(\cdot - a) N(\tilde{q}) & \text{ for $\tilde{q} \in \C_{\frac{k - 1}{2}}[x]$} \hspace{1 cm} \text{and}\\
				(\cdot - b) N(\tilde{q}) & \text{ for $\tilde{q} \in \C_{\frac{k - 1}{2}}[x]$}.
			\end{align*}
		\end{enumerate}
	\end{lem}
	\begin{proof}
		(See also \cite{Prestel}). By investigating the locations of roots of $q$, one sees that in both cases $q$ can be written in the form
		\begin{align*}
			N(\tilde{q}) \prod_{i = 1}^{l} (\cdot - y_{i}),
		\end{align*}
		where $0 \leq l \leq k$ is off the same parity as $k$, and $\tilde{q}$ is a complex polynomial (of degree $(k - l)/2$), and $y_{i} \in [a, b]$ for any $1 \leq i \leq l$. Now each factor of the form $(\cdot - y_{i})$ is a (non-negative) weighted average of $(\cdot - a)$ and $(\cdot - b)$, so $q$ can further written as a weighted average of polynomials of the form
		\begin{align*}
			N(\tilde{q}) (\cdot - a)^{l_{a}} (\cdot - b)^{l_{b}},
		\end{align*}
		where $l = l_{a} + l_{b}$ and $\tilde{q}$ as before. Finally, depending on the parities of $l_{a}$ and $l_{b}$, such polynomials can be rewritten in the previous four categories.
	\end{proof}
	By the previous lemma (choosing $a = x_{1}$ and $b = x_{k}$) one sees that (depending on the parity of $k$)
	\begin{align*}
		\langle ((z I - B)^{-1} - (z I - A)^{-1}) w, w\rangle
	\end{align*}
	can be either written as a sum of terms of the form
	\begin{align*}
		 \frac{N(\tilde{q})(z)}{(z - x_{2}) \cdots (z - x_{k})} \text{ and } \frac{N(\tilde{q})(z)}{(z - x_{1}) \cdots (z - x_{k - 1})}
	\end{align*}
	or
	\begin{align*}
		\frac{N(\tilde{q})(z)}{(z - x_{1}) \cdots (z - x_{k})} \text{ and } \frac{N(\tilde{q})(z)}{(z - x_{2}) \cdots (z - x_{k - 1})}
	\end{align*}
	where $\tilde{q}$ is a complex polynomial of suitable degree (as in the lemma). Consequently, as in the proof of Lemma \ref{main_identity}, $\langle (f(B) - f(A)) w, w\rangle$ can be written as a sum of terms of the form
	\begin{align*}
		[x_{0}, x_{1}, \ldots, x_{2 l - 1}]_{f N(q)},
	\end{align*}
	for some $l \geq 0$, $x_{1} < x_{2} < \ldots < x_{2 l - 1}$ on $F$ and $q \in \C_{l - 1}[x]$. We are done.

	``Moreover": As remarked before the proof, this follows from the proof of Lemma \ref{span_lemma}.
\end{proof}

Based on this characterization, one can give a generalization to Corollary \ref{local_property}.

\begin{kor}\label{general_local_property}
	Let $n \geq 1$ and $F_{1}, F_{2} \subset \R$ with the following property: if $F_{1} \setminus F_{2} \ni x < y \in F_{2} \setminus F_{1}$, then $(x, y) \cap F_{1} \cap F_{2}$ contains at least $2 n - 1$ points. Then if $f : F_{1} \cup F_{2} \to \R$ is such that $\restr{f}{F_{1}}$ and $\restr{f}{F_{2}}$ are both $n$-monotone, then so is $f$.
\end{kor}
\begin{proof}
	Pick any $1 \leq k \leq n$, $q \in \C_{k - 1}[x]$ and $x_{0} < x_{1} < \ldots < x_{2 k - 1} \in F_{1} \cup F_{2}$. By Theorem \ref{general_monotone} it suffices to check that
	\begin{align*}
		[x_{0}, x_{1}, \ldots, x_{2 k - 1}]_{f N(q)} \geq 0.
	\end{align*}
	By the assumption on the sets one may find a refinement of $(x_{0}, x_{1}, \ldots, x_{2 k - 1})$ for which all consequtive $2 n$ points belong completely to one of $F_{1}$ and $F_{2}$. But now Lemma \ref{refinement_lemma} implies the claim.
\end{proof}

\begin{huom}\label{general_regularity}
	While one could certainly formulate statements about regularity of $n$-monotone functions on general sets (akin to Corollary \ref{regularity}) in terms of divided differences, such line of thought is not pursued here.
\end{huom}

\begin{lause}\label{interpolation_failure}
	Let $n > 1$. Then there exists a set $F$, a $n$-monotone function $f : F \to \R$ and a point $x_{0} \in \conv(F) \setminus F$, such that $f$ cannot be extended to $n$-monotone function on $F \cup \{x_{0}\}$.
\end{lause}
\begin{proof}
	Let $F = \{x_{1} < x_{2} < \ldots < x_{2 n + 2} \}$ be any set of $(2 n + 2)$ points. We claim that one may choose $f$ to be a function on $F$ which agrees with a rational function $r_{1}$ on the set of first $2 n$ points and with another rational function $r_{2}$ on the set of last $2 n$ points; these rational functions are both of degree $n - 1$; and that $x_{0}$ can be chosen to be any point between $x_{n + 1}$ and $x_{n + 2}$.

	Indeed, choose any pairwise distinct $\lambda_{1}, \lambda_{2}, \ldots, \lambda_{2 n - 2}$ outside $\conv(F)$. Consider the function
	\begin{align*}
		z \mapsto \frac{(z - x_{3}) (z - x_{4}) \cdots (z - x_{2 n})}{(z - \lambda_{1}) (z - \lambda_{2}) \cdots (z - \lambda_{2 n - 2})}.
	\end{align*}
	It's easy to check that it has exactly $(n - 1)$ poles of both positive and negative residue so we may write it in the form $r_{2} - r_{1}$ where $r_{1}$ and $r_{2}$ are rational Pick functions of degree $n - 1$, unique up to constant. Since $r_{1}$ and $r_{2}$ agree on the middle $(2 n - 2)$ points, they determine $f$ as in the plan.

	To check that such $f$ is $n$-monotone, note that by the argument of Lemma \ref{refinement_lemma} and Theorem \ref{general_monotone} we only need to check that
	\begin{align*}
		[x_{1}, x_{2}, \ldots, x_{2 n}]_{f N(q)}, [x_{2}, x_{3}, \ldots, x_{2 n + 1}]_{f N(q)}, [x_{3}, x_{4}, \ldots, x_{2 n + 2}]_{f N(q)}
	\end{align*}
	are non-negative for any $q \in \C_{n - 1}[x]$. Since $f$ agrees with $r_{1}$ $(r_{2})$ on the first (last) $2 n$ points, the first and last terms are non-negative (by Theorem \ref{general_monotone}). Indeed both $r_{1}$ and $r_{2}$ are even $\infty$-monotone in the largest open interval containing $x_{i}$'s and not containing $\lambda_{i}$'s. Also, since \begin{align*}
		r_{2}(x_{2 n + 1}) - r_{1}(x_{2 n + 1}) = \frac{(x_{2 n + 1} - x_{3}) (x_{2 n + 1} - x_{4}) \cdots (x_{2 n + 1} - x_{2 n})}{(x_{2 n + 1} - \lambda_{1}) (x_{2 n + 1} - \lambda_{2}) \cdots (x_{2 n + 1} - \lambda_{2 n - 2})} > 0,
	\end{align*}
	and $f$ and $r_{1}$ agree on $\{x_{i}, 2 \leq i \leq 2 n\}$, we have
	\begin{align*}
		[x_{2}, x_{3}, \ldots, x_{2 n + 1}]_{f N(q)} \geq [x_{2}, x_{3}, \ldots, x_{2 n + 1}]_{r_{1} N(q)} \geq 0.
	\end{align*}
	$f$ is hence $n$-monotone.

	Let us now check that $f$ cannot be extended to any $x_{0} \in (x_{n + 1}, x_{n + 2})$. Indeed, assume that such extension exists; denote also the extension by $f$. Since $r_{1}$ is of degree $n - 1$, we may pick $q \in \C_{n - 1}[x]$ such that
	\begin{align*}
		0 &= [x_{1}, x_{2}, \ldots, x_{2 n}]_{r_{1} N(q)} = [x_{1}, x_{2}, \ldots, x_{2 n}]_{f N(q)}
	\end{align*}
	and $q(x_{0}) \neq 0$. Indeed, pick $q$ with the poles of $r_{1}$ as roots.

	On the other hand by Lemma \ref{refinement_lemma} we have for some $0 \leq t \leq 1$ (actually $t = (x_{2 n + 1} - x_{0})/(x_{2 n + 1} - x_{1})$)
	\begin{align*}
		[x_{1}, x_{2}, \ldots, x_{2 n}]_{f N(q)} &= t [x_{0}, x_{2}, \ldots, x_{2 n}]_{f N(q)} + (1 - t) [x_{0}, x_{1}, x_{2}, \ldots, x_{2 n - 1}]_{f N(q)}.
	\end{align*}
	But this forces at least one of $[x_{0}, x_{2}, \ldots, x_{2 n}]_{f N(q)}$ or $[x_{0}, x_{1}, x_{2}, \ldots, x_{2 n - 1}]_{f N(q)}$ to be zero, since they are both non-negative. Since $q(x_{0}) \neq 0$, in both cases there is an unique value for $f(x_{0})$, which makes the expression $0$. Since $r_{1}(x_{0})$ is such value, we must have $f(x_{0}) = r_{1}(x_{0})$.

	Running a symmetric argument with $r_{2}$ we see that we should also have $f(x_{0}) = r_{2}(x_{0})$, but since $r_{1}$ and $r_{2}$ do not agree outside $F$, this is impossible.
\end{proof}

While the previous result shows that the theory of $n$-monotone functions cannot be in general reduced to intervals, situation still collapses if the set $F$ is wide enough.

\begin{kor}
	(cf. \cite[Theorem 3.5]{Tom2}) Let $F \subset \R$ be set that is not bounded from below or above. Then if $f : F \to \R$ is $2$-monotone, $f$ is affine.
\end{kor}
\begin{proof}
	Take any three distinct points $x, y, z \in F$. Now since $f$ is $2$-monotone,
	\begin{align*}
		[x, y, z, M]_{f (\cdot - M)^{2}} = [x, y, z]_{f (\cdot - M)} = -M [x, y, z]_{f} + [x, y, z]_{(\cdot) f} \geq 0
	\end{align*}
	for every $M \in F \setminus \{x, y, z\}$. But since $M$ can attain arbitrarily small and large values, this is only possible if $[x, y, z]_{f} = 0$ on $F$, i.e. if $f$ is affine.
\end{proof}

\section{Acknowledgements}

I would like to thank Barry Simon for all the encouragement and fruitful discussions that lead to the birth of this paper, as well as sharing draft of his upcoming book on Loewner's theory \cite{Simon}, source of endless inspiration. In addition, I am deeply grateful to Eero Saksman for the continued support during the times of writing.

\bibliography{trondheim_1}

\begin{thebibliography}{10}

\bibitem{Ben}
J.~Bendat and S.~Sherman.
\newblock Monotone and convex operator functions.
\newblock {\em Trans. Amer. Math. Soc.}, 79:58--71, 1955.

\bibitem{Bullen}
P.~Bullen.
\newblock A criterion for n-convexity.
\newblock {\em Pacific Journal of Mathematics}, 36(1):81--98, 1971.

\bibitem{Chandler}
J.~Chandler.
\newblock Extensions of monotone operator functions.
\newblock {\em Proceedings of the American Mathematical Society},
  54(1):221--224, 1976.

\bibitem{Boo}
C.~de~Boor.
\newblock Divided differences.
\newblock {\em Surv. Approx. Theory}, 1:46--69, 2005.

\bibitem{Dob}
O.~Dobsch.
\newblock Matrixfunktionen beschr\"ankter {S}chwankung.
\newblock {\em Math. Z.}, 43(1):353--388, 1938.

\bibitem{Don_gen}
W.~F. Donoghue.
\newblock Monotone operator functions on arbitrary sets.
\newblock {\em Proceedings of the American Mathematical Society}, 78(1):93--96,
  1980.

\bibitem{Don}
W.~F. Donoghue, Jr.
\newblock {\em Monotone matrix functions and analytic continuation}.
\newblock Springer-Verlag, New York-Heidelberg, 1974.
\newblock Die Grundlehren der mathematischen Wissenschaften, Band 207.

\bibitem{Don_gen2}
W.~F. Donoghue~Jr.
\newblock Another extension of loewner's theorem.
\newblock {\em Journal of mathematical analysis and applications},
  110(2):323--326, 1985.

\bibitem{Tom}
F.~Hansen and J.~Tomiyama.
\newblock Differential analysis of matrix convex functions.
\newblock {\em Linear Algebra Appl.}, 420(1):102--116, 2007.

\bibitem{Tom2}
F.~Hansen and J.~Tomiyama.
\newblock Differential analysis of matrix convex functions. {II}.
\newblock {\em JIPAM. J. Inequal. Pure Appl. Math.}, 10(2):Article 32, 5, 2009.

\bibitem{Heina}
O.~Hein{\"a}vaara.
\newblock Local characterizations for the matrix monotonicity and convexity of
  fixed order.
\newblock {\em Proceedings of the American Mathematical Society},
  146(9):3791--3799, 2018.

\bibitem{Hiai}
F.~Hiai.
\newblock Matrix analysis: matrix monotone functions, matrix means, and
  majorization.
\newblock {\em Interdisciplinary Information Sciences}, 16(2):139--248, 2010.

\bibitem{Kraus}
F.~Kraus.
\newblock \"{U}ber konvexe {M}atrixfunktionen.
\newblock {\em Math. Z.}, 41(1):18--42, 1936.

\bibitem{Low}
K.~L{\"o}wner.
\newblock \"{U}ber monotone {M}atrixfunktionen.
\newblock {\em Math. Z.}, 38(1):177--216, 1934.

\bibitem{Prestel}
A.~Prestel and C.~Delzell.
\newblock {\em Positive polynomials: from Hilbert’s 17th problem to real
  algebra}.
\newblock Springer Science \& Business Media, 2013.

\bibitem{Rosen}
M.~Rosenblum and J.~Rovnyak.
\newblock An operator-theoretic approach to theorems of the pick-nevanlinna and
  loewner types. i.
\newblock {\em Integral Equations and Operator Theory}, 3(3):408--436, 1980.

\bibitem{Simon}
B.~Simon.
\newblock {\em Loewner’s Theorem on Monotone Matrix Functions}.
\newblock Springer, to appear, 2019.

\bibitem{Smulj}
J.~L. {\v{S}}mul'jan.
\newblock Monotone operator functions on a set consisting of an interval and a
  point.
\newblock {\em Ukrain. Mat. {\v{Z}}}, 17:130--136, 1965.

\end{thebibliography}
\bibliographystyle{abbrv}

\end{document}